\documentclass{AIMS}
\usepackage{amsmath}
  \usepackage{paralist}
  \usepackage{graphics} 
  \usepackage{epsfig, epsf} 
\usepackage{graphicx, graphics, multicol, multirow, setspace}  \usepackage{epstopdf}

\def\currentvolume{X}
 \def\currentissue{X}
  \def\currentyear{200X}
   \def\currentmonth{XX}
    \def\ppages{X--XX}
     \def\DOI{10.3934/xx.xx.xx.xx}

\newtheorem{theorem}{Theorem}[section]
\newtheorem{corollary}{Corollary}

\newtheorem{lemma}[theorem]{Lemma}

\theoremstyle{definition}

\newtheorem{remark}{Remark}

\newcommand{\Potential}{\mathcal{W}}

\newcommand{\wbtw}{\Potential_{\textrm{\begin{tiny}{BTW}\end{tiny}}}}
\newcommand{\bx}{\boldsymbol{x}}
\newcommand{\by}{\boldsymbol{y}}

\newcommand{\be}{\boldsymbol{e}}

\newcommand{\tr}{{\textrm{tr}}}
\newcommand{\bvphi}{\boldsymbol{\varphi}}
\newcommand{\bvpsi}{\boldsymbol{\psi}}
\newcommand{\bn}{\boldsymbol{n}}
\newcommand{\bm}{\boldsymbol{m}}
\newcommand{\wconv}{\rightharpoonup}
\newcommand{\adj}{\operatorname{adj}}
\newcommand{\dist}{\operatorname{dist}}
\newcommand{\Feff}{F_{\textrm{\tiny{eff}}}}
\newcommand{\lamin}{\lambda_{\textrm{\tiny{min}}}}
\newcommand{\lamax}{\lambda_{\textrm{\tiny{max}}}}

\makeatletter
\numberwithin{equation}{section}
\makeatother

\title[A Landau--de Gennes theory of  liquid crystal elastomers]
      {A Landau--de Gennes theory of  liquid crystal elastomers}

\author[M.~Carme Calderer, Carlos  A.Garavito Garz{\'o}n, and Baisheng  Yan]{}


 \email{mcc@math.umn.edu}
 \email{garav007@umn.edu}
 \email{yan@math.msu.edu}

\begin{document}
\maketitle

\centerline{\scshape M.~Carme Calderer \and Carlos  A.Garavito Garz{\'o}n}
\medskip
{\footnotesize
 \centerline{School of Mathematics}
   \centerline{University of Minnesota}
   \centerline{206 Church Street S.E.}
\centerline{Minneapolis, MN 55455, USA.}
} 

\medskip

\centerline{\scshape Baisheng Yan}
\medskip
{\footnotesize
 \centerline{Department of Mathematics}
   \centerline{Michigan State University}
\centerline{619 Red Cedar Road}
   \centerline{East Lansing, MI 48824, USA}
}

\bigskip


\begin{abstract}
In this article, we  study minimization of the energy of a Landau-de Gennes liquid crystal elastomer.
 The total energy consists of the sum of the Lagrangian elastic stored  energy function of the elastomer and the Eulerian Landau-de Gennes energy of the liquid crystal. 

There are two related sources of anisotropy in the model, that of the rigid units represented by the traceless nematic {\it order tensor} $Q$, and the positive definite {\it step-length  tensor} $L$ characterizing the anisotropy of the network. 
This work is motivated by the study of cytoskeletal networks which can be regarded  as consisting of rigid rod units crosslinked into a polymeric-type network. 
Due to the mixed Eulerian-Lagrangian structure of the energy, it is essential that the deformation maps $\bvphi$ be invertible.  For this, we require  sufficient regularity of   the fields $(\bvphi, Q)$ of the problem, and that   the deformation map satisfies the Ciarlet-Ne{\v c}as injectivity condition. These, in turn, determine what boundary conditions  are admissible, which include  the case of  Dirichlet conditions on  both fields. Alternatively, the approach of including the Rapini-Papoular surface energy for the {\it pull-back} tensor $\tilde Q$ is also discussed.  The regularity requirements also lead us to consider  powers of the gradient of the order tensor $Q$ higher than quadratic in the energy.


We assume polyconvexity of the stored energy function with respect to the effective deformation tensor and apply methods of calculus of variations from isotropic nonlinear elasticity. 
Recovery of minimizing sequences of deformation gradients from the corresponding sequences of effective deformation tensors  requires invertibility of the anisotropic shape tensor $L$. We formulate a necessary and sufficient condition to guarantee this invertibility property  in terms of the growth to infinity of the bulk liquid crystal energy $f(Q)$, as the minimum eigenvalue of $Q$ approaches the singular limit of $-\frac{1}{3}$. It turns out that $L$ becomes singular as the minimum eigenvalue of $Q$ reaches $-\frac{1}{3}$. Lower bounds on the eigenvalues of $Q$ are needed to ensure compatibility between the theories of  Landau-de Gennes and Maier-Saupe of nematics \cite{BallMajumdar2010}.

\end{abstract}

\section{Introduction}

We investigate  existence of minimizers of the energy of a liquid crystal elastomer that in the reference configuration  occupies a domain $\Omega\subset {\mathbf R}^3$. 
We assume that the total energy consists of the Lagrangian contribution of the anisotropic elastomer and the  Landau-de Gennes liquid crystal energy in the  Eulerian frame.
The mixed Lagrangian-Eulerian setting of the problem requires injectivity of the deformation map. For this, we require the admissible deformation maps to have sufficent regularity and 
 to satisfy the Ciarlet-Necas condition as well. 
This work aims at investigation the coupling of  elastic energy and nematic order, and it applies  to, both,  thermotropic and lyotropic systems. 
 
Liquid crystal elastomers are  anisotropic nonlinear elastic materials, with the source of anisotropy stemming from  the presence of elongated  rigid monomer units,
 that are  either inserted in the back-bone as part of the polymer main-chain or  are present as side groups. 
They  are elastic solids that may also present fluid and mixed regimes \cite{conti2002soft}, \cite{desimone-doltzman2000}, \cite{desimone-doltzman2002},
 \cite{fried-sellers2006}, \cite{warner2007liquid}.
In nematic fluids, molecules tend to align themselves along preferential directions but do not present ordering of centers of mass.

The nature of the connections between polymer chains and rigid monomer units plays a main role in determining the behavior of liquid crystal elastomers \cite{warner2007liquid}. 
From a different perspective,  studies of actin and cytoskeletal networks (\cite{CaGaLuo2013}, \cite{Gardel2004}, \cite{wagner-cytoskeletal2006}) show  lyotropic rigid rod systems 
crosslinked into  networks by elastomer chains or linkers that present qualitative properties of liquid crystal elastomoers. 
A main feature of these networks is the  average number of connections between rods and the location of these connections in the rod. 
These motivates us to consider two limiting types of lyotropic systems, the first one corresponding to  a  nematic liquid with rods weakly coupled into the network. 
   The second case is that of a solid where rod rotation occurs uniquely as a result of elastic deformation, such as an elastomer made of material fibers.  
Consequeantly, we postulate Eulerian and Lagrangian Landau-de Gennes liquid crystal energies, respectively, for these  systems.
Many physical systems are found having intermediate properties between
 these two limiting cases, and  it, then, may be appropriate to postulate the energy as  a weighted sum of the Eulerian and Lagrangian energies, scaled according to a
 macroscopic parameter representing the density of crosslinks. The analogous property holds for the elastic energy of the system that we discuss in this article. 

The study of a lyotropic elastomer with Lagrangian liquid crystal energy was carried out in previous work where we also analyzed liquid crystal
 phase transitions triggered by change in rod density \cite{CaGaLuo2013}. In this work, we 
assume that the anisotropic behavior of the rigid units is represented by a liquid-like, Eulerian liquid crystal energy. 
Sufficient regularity of the deformation map is required to guarantee its invertibility,  in order to pass from the current configuration of the nematic liquid crystal 
 to the reference configuration of the elastic solid.

There are two main quantities that describe the anisotropy associated with a liquid crystal elastomer:  the traceless  order tensor $Q$ describing the nematic order of rigid rod units,  
and the positive definite step-length tensor $L$  
that encodes the shape of the network. $L$ is spherical for isotropic polymers and spheroidal for uniaxial nematic
elastomers, in which case, it has eigenvalues $l_\|$ and $l_{\perp}$ (double).
The quantity $r:= \frac{l_{\|}}{l_\perp}-1$ measures the degree of anisotropy of the network, with positive values corresponding to
prolate shape and negative ones to  oblate. In the prolate geometry, the eigenvector $\bn$
associated with  $l_{\|}$ is the director of the theory, giving the average direction of alignment of the rods and also
the direction of shape elongation of the network. It is natural to assume that $L$ and $Q$ share the same eigenvectors. In this article, we set $Q= L- \frac{1}{3}(\tr L) I$ so as 
to obtain a traceless $Q$, from a given positive definite tensor $L$ (\cite{warner2007liquid}, page 49).
The free energy may also carry information on the anisotropy $L_0$   imprinted in the network  at crosslinking the original polymer melt. In this work, we choose $L_0=I$.

The standard Landau-de Gennes free energy density consists of  the sum of scalar quadratic terms of $\nabla Q$ and the bulk  scalar function $f(Q)$, usually a polynomial  of the trace of
 powers of $Q$,   describing  the phase transition between the isotropic and the nematic  phases of the liquid crystal \cite{longa1986}, \cite {MottramNewton2004}. In this work, we take a departure from these forms
but still keeping consistentcy with the original Landau-de Gennes theory. 
We first observe that  the polynomial growth is not physically realistic, since
it is expected that the energy should grow unboundedly  near  limiting alignment configurations. In the uniaxial case, these correspond to perfect alignment, characterized by 
scalar order parameter $s=1$, and configurations where the rods are confined to the plane perpendicular to the director $\bn$  \cite{calderer-liu2000}, \cite{ericksen1991liquid}. 
In terms of the order tensor, the limiting configuration is represented by the minimum eigenvalue taking the value $-\frac{1}{3}$.  Bounds on eigenvalues of $Q$ are not part of the original 
theory, but are needed for it to be compatible with the Maier-Saupe theory from statistical  physics \cite{BallMajumdar2010}. 
 This turns out to be as well an essential element of our analysis. 
Also,  we take the gradient of  $Q$ with respect to  space variables  in the deformed configuration, in which case the coupling with the deformation gradient $F$ emerges naturally. It turns out that  powers of gradient of $Q$, higher than quadratic,  in the 
energy,  are required for compactness.

The elastic energy density proposed by
Blandon, Terentjev and Warner  is given by the trace form \cite{warner2007liquid}:
\begin{equation}
\wbtw= \mu \tr (L_0F^TL^{-1}F- \frac{1}{3}I). \label{btwL0}
\end{equation}
This is the analog of the Neo-Hookean energy of isotropic
elasticity  and is also derived from Gaussian statistical mechanics. Let us call $G=L^{-\frac{1}{2}}FL_0^{\frac{1}{2}}$ the effective deformation   tensor of the network; of course,  
in general,  $G$ is not itself  a gradient. Then $\wbtw =\mu (|G|^2-1).$ Motivated by the theory of existence of minimizers of
isotropic nonlinear elasticity (\cite{ball77}), in his PhD, thesis \cite{chongluo2010},   Luo  generalizes  $\wbtw$ to  polyconvex  stored energy density functions
$\hat w(G(\bx))$; that is, $\hat w(G)= \Psi(G, \adj G, \det G)$ is a convex function of $(G, \adj G, \det G)$.  This approach was used later in  \cite{CaGaLuo2013}) to model 
phase transitions in rod networks.    In order to recover the limiting deformation gradient $F^*$ from the minimizing sequences $\{G_k\}_{k\geq 1}$, it is necessary that the  
minimizing sequences $\{L_k\}$ yield a nonsingular limit.  This is achieved     
 by requiring the blowup of  $f(Q)$ as the minimum eigenvalue of $Q$ tends to $-\frac{1}{3}$,  that is, $f(Q)\to \infty $ as $\det(Q+\frac{1}{3}I)=\det L\to 0$. We point out that this lends another significance 
to the minimum eigenvalue limit. 

  In a related work \cite{calderer-luo2012}, Calderer and Luo carried out a mixed finite element analysis of the elastomer 
trace energy coupled with that of  the Ericksen model of  uniaxial nematic liquid crystals. This was later used in  numerical simulations of domain formation in two-dimensional extensional
 deformations. From a different persepective, a study of uniaxial elastomers with variable length director was carried out in \cite{CaldererLiuYan2006}, for a restricted set of deformation maps.

Let us now comment on the deeper mechanical significance of $G$ (also denoted $\Feff$) in connection with the special class of spontaneous deformation  as  brought up by Agostiniani,
 DeSimone and Teresi  \cite{desimone-agostiniani2012}, \cite{desimone2009elastic}.  Spontaneous deformations are those that do not cost any elastic energy.  
In uniaxial nematics, they correspond to volume preserving uniaxial extensions along the eigenvector $\bn$, the latter being defined in the current configuration of the elastomer. 
 (In the context of  biaxial nematics, volume preserving biaxial extensions are also spontaneous deformations represented by $F_s=\sqrt{L_s}$, where, in its (unit) eigenvector
 representation, $L_s:= a\bm\otimes\bm+ \mathbf r\otimes\mathbf r)+ c\bn\otimes\bn$, $ \, a, b, c>0,$  \, $abc=1$).  Spontaneous strains have crystallographic significance in that 
they represent variants in domain patterns present in low energy microstructure. The chevron domains observed by Sanchez and Finkelman in liquid crystal elastomers realize
 the microstructure in these materials. \cite{KundlerFinkelmann1995}, \cite{Finkelmann-Sanchez2008} and \cite{zubarev1999monodomain}.

 We note that  $\Feff$ does not represent the deformation of the elastomer with respect to the identity but with  respect to spontaneous deformations $\sqrt{L_s}$. 
Consequently,  the Trieste group also pointed out that $\Feff$ is the only deformation tensor that contributes to the elastic energy of the elastomer,   this energy being   
isotropic, since it has the same form regardless the current value of $\bn$.    These  authors proposed to include as well an anisotropic component in the energy depending on
 $F^TF$ (\ref{des-12}). We observe the regularizing role of the latter part of the energy. Without loss of generality,  in this work, we neglect the latter term and address 
the more challenging problem  of the energy depending only on $G$ ($\Feff$).

In addition to the trace models of liquid crystal elastomer energy studied by Terentjev and Warner (\cite{warner2007liquid} and references therein), 
 generalizations of these earlier forms have been proposed and studied by several authors (\cite{anderson-carlson-fried1999} and \cite{fried-sellers2003};   \cite{desimone-agostiniani2012}, \cite{cesana2008strain}, \cite{desimone-cesana2011} and \cite{desimone2009elastic}). In these references, the authors propose energies based on powers of the earlier trace form, including Ogden type energies,  and study their extensions to account for semisoft elasticity.  
Also, the analysis of equilibrium states presented in \cite{desimone-agostiniani2012} applies to  elastomer energy density functions that  are not quasiconvex, these being appropriate 
  to model crystal-like phase transitions.
  


The types of  boundary conditions that we consider include prescribing the deformation map on the boundary or part of it. In such cases, we also prescribe components of  the  order tensor $Q$  there.
In an alternate approach, instead of prescribing boundary conditions on $Q$, we include a surface free energy of the Rapini-Papoular type.  The latter also extends to those parts of the  boundary with no prescribed Dirichlet conditions on the deformation map. In such a case,  the energy integral is formulated in terms of the {\it{pull-back}} order tensor $\tilde Q(\bx)$, $\bx\in\Omega$.

This article is organized as follows.   Section 2 is devoted to  modeling, with special emphasis on analyzing the type of coupling between the Landau-de Gennes model of nematic
 liquid crystal and the nonlinear elasticity of the anisotropic network.  In Section 3, we study the admissible set of fields corresponding to finite  energy, focusing on classes of 
 deformation maps and boundary conditions for which the invertibility  property of  the maps holds.  Section 4 addresses minimization of the energy.  
  The conclusions are described in Section 5.

\section{The  Landau-de Gennes  liquid crystal elastomer} 
Equilibrium configurations of  nematic liquid crystal elastomers are characterized by the
 deformation gradient $F$ together with the symmetric tensors $L$ and $Q$,
  describing the shape of the material and the nematic order, respectively.
  Within the point of view of the mean-field theory, the state of alignment of a nematic liquid crystal is given by a probability distribution function $\rho$ in the unit sphere. The order tensor is defined as the second order moment of $\rho$:
  \begin{equation}
  Q= \int_{{\mathbb S}^2} (\bm\otimes\bm -\frac{1}{3} I)\rho(\bm)\,d\bm. \label{order-tensor}
  \end{equation}
  From this definition, it follows that $Q$ is a symmetric, traceless, $3\times 3$ matrix with bounded eigenvalues   (\cite{Forest2004} and \cite{BallMajumdar2010})
   \begin{equation}
-\frac{1}{3}\leq \lambda_i(Q)\leq \frac{2}{3}, \,\,\,  i=1, 2, 3, \quad \sum_{i=1}^3\lambda_i(Q)=0. \label{eigenvaluesQ}
   \end{equation}
 For  biaxial nematic, $Q$ admits the representation \begin{equation}Q=r (\be_1\otimes\be_1-\frac{1}{3}I)+ s(\be_2\otimes\be_2-\frac{1}{3}I), \label{Q-biaxial}\end{equation}
where $r$ and $s$ correspond to the biaxial order parameters
$$s= \lambda_1-\lambda_3= 2\lambda_1+\lambda_2, \quad r= \lambda_2-\lambda_3=\lambda_1+2\lambda_2, $$
and $\be_1$ and $\be_2$ are unit eigenvectors. 
If $r=0$, then (\ref{Q-biaxial})  yields   the uniaxial order tensor
$Q= s(\be_2\otimes\be_2-\frac{1}{3}I),$
where $s\in(-\frac{1}{2}, 1)$ corresponds to the uniaxial nematic order parameter and $\bn=\mathbf e_2$ is the unit  director field of the theory.

In a rigorous study of the Landau-De Gennes model \cite{Apala-Majumdar2010},  Majumdar observes that the definition of $Q$ given by (\ref{order-tensor}) is that of the Maier-Saupe mean-field theory \cite{maier-saupe}.  In the original theory by Landau and de Gennes \cite{DeGennes1995physics}, \cite{longa1986}, $Q$ has a phenomenological role as a  dielectric or magnetic susceptibility tensor, and its  eigenvalues   do not satisfy any inequality constraints. It is the latter that  bring compatibility to these theories. 
 
     Following   the property of  freely  joined rods, we assume that  $L$ and $Q$
have common eigenvectors and propose the constitutive relation
\begin{equation}
L=a_0(Q+\frac{1}{3} I), \label{LQ0}
\end{equation}
where $a_0=\tr L> 0$ is constant.  The linear constitutive equation (\ref{LQ0}) is analogous to those proposed by  Terentjev and Warner \cite{warner2007liquid} and Fried and Sellers \cite{fried-sellers2003} stating that, given a symmetric and traceless tensor $Q$ and a  constant $\beta>0$, there is a one  $\alpha$-parameter family of step-length tensors $L$ with $\tr L=\beta$,  and such that $L= \beta(\alpha Q+ \frac{1}{3}I).$
The form (\ref{LQ0}) corresponds to taking $\alpha=1$ and $\beta=a_0.$

\begin{figure}
\centering
\scalebox{.6}{
\label{isotropic-nematic}
\includegraphics{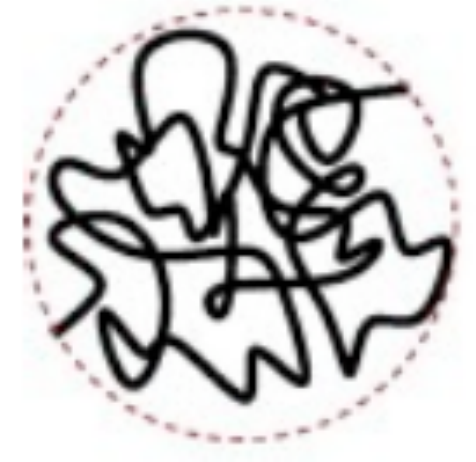}\quad \quad \quad \quad \quad \quad \quad \quad \quad \quad \quad 
\includegraphics{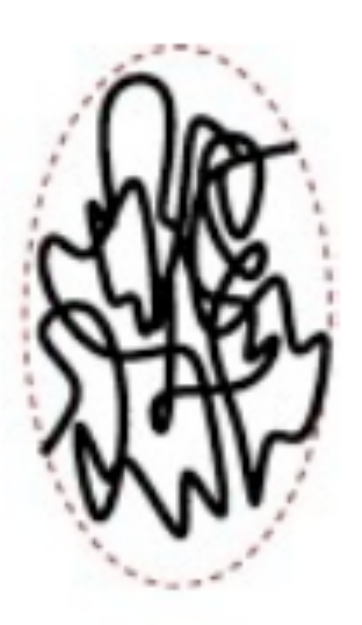}}
\caption{Representation of the step-lenght tensor $L$:  isotropic polymer (left) and nematic polymer (right).}
\end{figure}
Letting $l_1,l_2,l_3$ denote the   eigenvalues of $L$,  it  reduces to the uniaxial nematic with director $\bn$ and order parameter a multiple of $s$
in the case that
\[
l_2=l_1:=l_\perp, \,\, l_3:=l_\|.
\]
The mechanical response of the elastomer along $\bn$ is distinguished from that along any of the transverse directions.  

From the constitutive assumption (\ref{LQ0}) it follows that
\begin{equation}
\det L=0 \Longleftrightarrow \det(Q+\frac{1}{3}I)=0  \Longleftrightarrow \lamin(Q)=-\frac{1}{3}, \label{Linvertibility}
\end{equation}
where $\lamin(Q)$ stands for the minimal eigenvalue of a symmetric tensor $Q$. This shows another consequence of requiring $\lamin(Q)>-\frac{1}{3}$: to gurantee the invertibility of $L$, and so, to be able to 
recover the gradient of deformation from the effective deformation tensor $G$. 

We assume that, in the reference configuration, a liquid crystal elastomer occupies a bounded  domain
$\Omega\subset {\mathbf R}^3$ with smooth boundary $\partial\Omega$.
We denote the deformation map of the elastomer and its gradient as
\begin{eqnarray}
&\bvphi: \Omega\to \bvphi(\Omega), \quad \by=\bvphi(\bx),&
 \label{deformation} \\ 
& F(\bx)=\nabla\bvphi(\bx),
  \quad \det F(\bx)\ge \delta_0,& \label {incompressible}
\end{eqnarray}
where $\delta_0>0$ is a given constant. This last constraint on  the determinant expresses the maximum degree of compressibility allowed to the material.

The   trace-form    free energy  density $\wbtw= \mu \tr (L_0F^TL^{-1}F- \frac{1}{3}I)$  expresses the coupling between the step length tensor $L$ and the deformation gradient $F$ and also encodes the anisotropy $L_0$ of the reference configuration.
Usually, $L_0$ is a constant positive definite symmetric tensor  and, by rescaling,  in what follows, we assume $L_0=I.$ Define $G=\Feff=L^{-\frac12} F$.

In general, scalar functions of the invariants of the tensors $FF^T$ and $\Feff^T\Feff$ are admissible.
Let us examine how vectors transform in each of these cases as illustrated in figure (\ref{effective-deformation-tensor}).
\begin{figure}
\label{effective-deformation-tensor}
\centerline{
\scalebox{0.5}
{\includegraphics{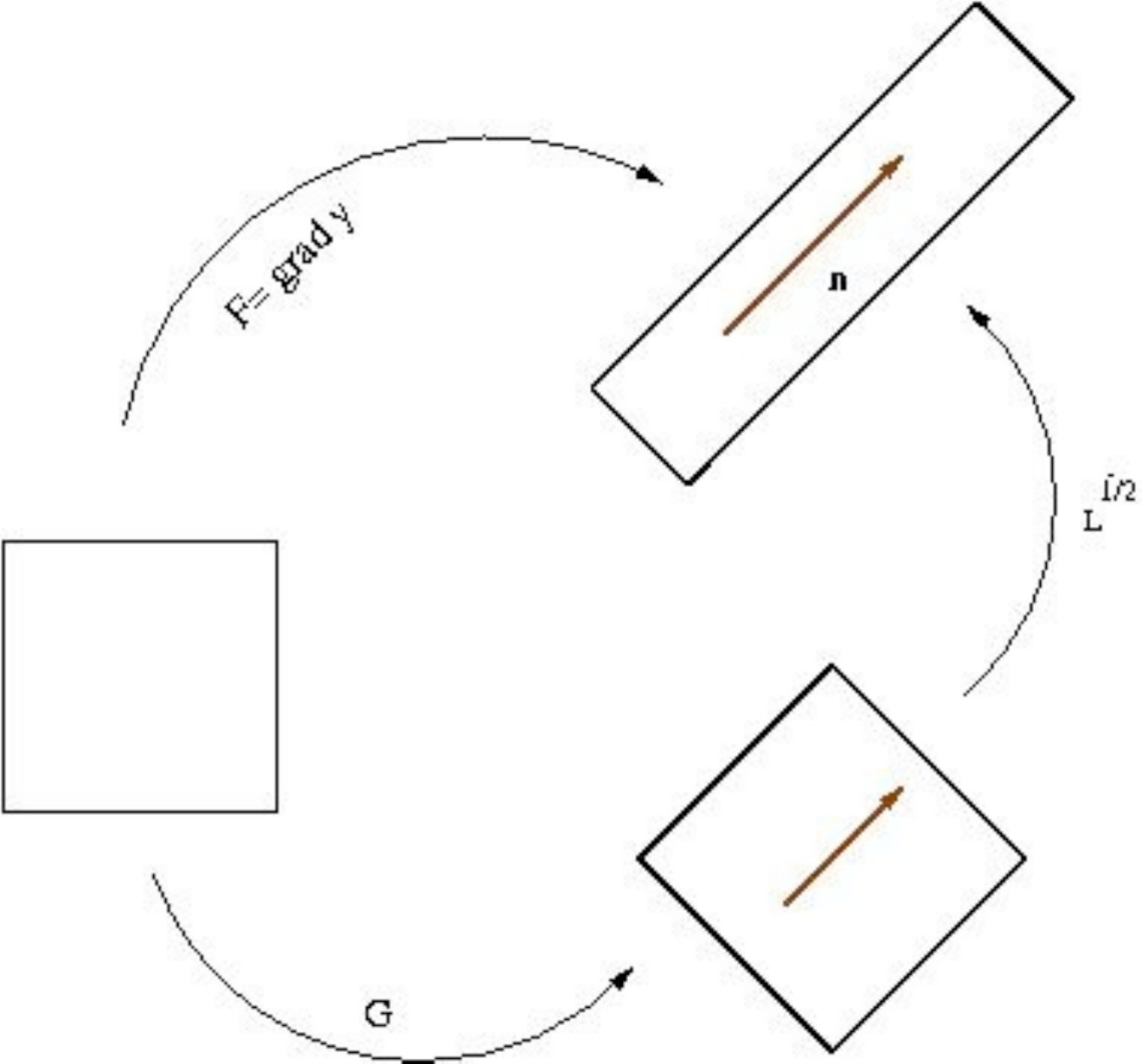}}}
\caption{Relation between the deformation gradient $F$ and the effective deformation tensor $\Feff$. }
\end{figure}
In  \cite{desimone2009elastic}, the authors propose an elastomer energy of the form
\begin{equation}\label{des-12}
W(F, L, L_0)= \alpha W_{\alpha}(\Feff\Feff^T) + \beta W_{\beta}(F^TF), 
\end{equation}
with $\Feff$ representing the deformation tensor with respect to spontaneous deformations. The $\alpha$-component of this energy was also analyzed in previous work   \cite{CaGaLuo2013}, \cite{chongluo2010}, and in 
 \cite{CaldererLiuYan2006} and 
\cite{calderer-luo2012} for  the Neo-Hookean trace form of the energy. 
As indicated in the introduction, these terms correspond to  two limiting liquid crystal elastomer behaviors, perhaps ideal, the first corresponding to
the case that the material is made of pure fibers and the second  to a standard liquid crystal elastomer but  showing interaction between alignment and deformation.

We focus on the cases of nontrivial coupling between  $Q$ and $F$ and propose a Landau-de Gennes elastomer energy of the following  form:
\begin{equation}
\mathcal E(\bvphi,Q)=  \int_{\Omega} \hat W(G)\,d\bx+ \int_{\bvphi(\Omega)} \big(\mathcal L({\nabla}_{\by} Q, Q) + f(Q)\big)\,d\by, \label{total-energy-Lag_Eul}
\end{equation}
where, for certain given $\bvphi$ and $Q$,
\begin{eqnarray}
&L(\by)=a_0 (Q(\by)+\frac13 I), \,\, \mbox{where $a_0>0$ is a constant,} & \label{eq210}\\
&G= G(\bx) :=\tilde L(\bx)^{-\frac{1}{2}}F, \,\, F=\nabla \bvphi(\bx),\, \,  \tilde L(\bx)=L(\bvphi(\bx)), &\label{eq29}\\
&(\nabla_{\by}Q)_{ijk}=\frac{\partial Q_{ij}}{\partial y_k}, \,\, 1\leq i, j, k\leq 3. &\label{eq211}
\end{eqnarray}

In what follows, we denote ${\mathbb M}^3$  the space of three-dimensional tensors and $\mathbb{M}^3_+=\{M\in \mathbb{M}^3:\det M>0\}.$ According to (\ref{eigenvaluesQ}),  we define
\begin{equation}\label{set-Q}
\mathcal Q=\{Q\in \mathbb{M}^3: Q=Q^T,\,\, \tr Q=0,\,\, -\frac13<\lamin(Q)\le \lamax(Q)\le \frac 23\}.
\end{equation}
 Then, for $Q\in \mathcal Q$, it follows that $|Q|\le \frac{2}{\sqrt{3}}$ and the tensor $L$ defined by (\ref{LQ0}) is invertible.

As for the density functions $\hat W$, $\mathcal L$ and $f$ in (\ref{total-energy-Lag_Eul}),  we make the following assumptions motivated by the analogous ones in isotropic nonlinear elasticity
\cite{ball77}.
\smallskip

\noindent
{\it Polyconvexity and coerciveness of $\hat W$}: There exists a convex function $\Psi:
\mathbb{M}^3_{+}\times \mathbb{M}^3\times \mathbb{R}^+\to
\mathbb{R}$
such that  $\hat W$ in
(\ref{total-energy-Lag_Eul}) satisfies
 \begin{equation}
\hat W(G)= \Psi(G, \adj G, \det G). \label{polyconvexity}
 \end{equation}
Also, there exist constants $\alpha>0,  \, p>3 $ such that
\begin{eqnarray}\label{coer-W}
  \hat W(G)\geq \alpha |G|^p ,   \quad\forall\, G\in
\mathbb{M}^3_{+}.
\label{coerciveness}\end{eqnarray}

\noindent
{\it Convexity and growth of $\mathcal L$}:  The Landau-de Gennes energy function $\mathcal L(\nabla_{\by} Q,Q)$ is convex in $\nabla_{\by} Q$. Moreover, there exists a constant  $\kappa>0$  
such that
\begin{equation}\mathcal L(\nabla_{\by} Q,Q)\geq \kappa|\nabla_{\by} Q|^{r}, \label{supergrowth} \end{equation}
where $r$ is a constant satisfying
\begin{equation}\label{cond-r}
r>\max\{3,\,\frac{p}{p-3}\}. 
\end{equation}

\noindent
{\it Blow-up of  $f$}:
The bulk free energy density
  $f: \mathcal Q\to {\mathbb R}^+ $ is continuous and satisfies
 \begin{equation}\lim_{\lamin(Q)\to -\frac{1}{3}}f(Q)=+\infty. \label{Phi3}
\end{equation}

\noindent
{\bf Notation.} For $r$ as in (\ref{cond-r}), we define 
\begin{equation}
q=\frac{pr}{p+r}. \label{q}
\end{equation}
We point out that  $q>\max\{\frac{p}{p-2},\frac{3p}{p+3}\}>1$.

\noindent
{\bf Remarks. \,}
\begin{enumerate}
\item If $L$ is given by  (\ref{LQ0}), then $0<\det L\le (a_0)^3$ when $Q\in \mathcal Q.$ So, in the case that $\det F\ge \delta_0$,  we have that
$\det G\geq \delta_0 (a_0)^{-\frac{3}{2}}$. Therefore, no condition on the growth near zero-determinant has to be imposed on $\hat W(G)$.
\item We observe that  $\det L$ can become arbitrarily small.
This corresponds to the polymer adopting a needle or plate shape.
\item The growth condition (\ref{Phi3}) has been applied in \cite{BallMajumdar2010} in the context of  studying of minimization of the Landau-de Gennes energy. Its restriction to the uniaxial case was first proposed by Ericksen \cite{ericksen1991liquid} and used in analysis of defects in liquid crystal flow \cite{CM96, CM97, CM98, calderer-liu2000}.
\end{enumerate}
In order to achieve a better understanding of the gradient part of the of the Landau-de Gennes energy (\ref{supergrowth}), let us give a brief review of  the standard  liquid crystal theory \cite{MottramNewton2004}. The total energy is of the form
\begin{eqnarray*}
&&\mathcal E_{\textrm{\tiny{LdeG}}}= \int_{\Omega}(\Psi(\nabla Q, Q) + f_{\textrm{B}}(Q; T))\, d\bx, \\
&& \Psi(\nabla Q, Q) = \sum_{i=1}^4 L_i I_i, \label{multi-constant} \\
&& f_{\textrm{B}}(Q; T)= \frac{a(T)}{2} \tr Q^2-\frac{b}{3}\tr Q^3+\frac{c}{4}\tr Q^4, 
\end{eqnarray*}
where $a(T)=\alpha(T-T^*), $  $T>0$ denotes the absolute temperature, $\alpha, T^*, b, c, L_i=L_i(T)$ are constants, and 
\begin{eqnarray*}
&&I_1=Q_{ij,j}Q_{ik,k}, \,\, I_2=Q_{ik,j}Q_{ij.k}, \\
&& I_3=Q_{ij,k}Q_{ij,k}, \, \, I_4=Q_{lk}Q_{ij.l}Q_{ij.k}. 
\end{eqnarray*}
In the special case of a single constant,  the energy (\ref{multi-constant}) reduces to $L|\nabla Q|^2$. For this energy,  existence of  global minimizer  was discussed in \cite{majumdar2010landau}, with further studies of  regularity, characterization of uniaxial and biaxial states,  and structure of defect sets presented in \cite{Majumdar2012}. 

\section{Admissible classes of fields} In what follows,we let $p>3$ and  $r>\max\{3,\frac{p}{p-3}\}$ be as in (\ref{cond-r}).  We consider the admissible classes of $(\bvphi,Q)$ for the energy $\mathcal E(\bvphi,Q)$,  from subclasses of functions $\bvphi\in W^{1,p}(\Omega,\mathbb{R}^3)$ and $Q\in W^{1,r}(\bvphi(\Omega),\mathbb{M}^3).$ Throughout the paper, given a measurable set $B\subset \mathbb{R}^3$ and a measurable function $\omega$ on $B$, we define $\omega\in W^{1,s}(B)$, provided there exist an open set $O$ containing $B$ and a function $\tilde \omega\in W^{1,s}(O)$ such that $\tilde \omega=\omega$ on $B$.

 We assume $\bvphi\in W^{1,p}(\Omega,\mathbb{R}^3)$ and  that  the condition (\ref{incompressible}) above is satisfied. To avoid material inter-penetration,  we also require that $\bvphi$ be injective. The latter issue is  addressed in  forthcoming lemmas. 

Since $p>3$, by Sobolev embedding, every map $\bvphi\in W^{1,p}(\Omega,\mathbb{R}^3)$ is H\"older continuous on $\bar\Omega$. Moreover,  in \cite{Marcus-Mizel} Marcus and Mizel  proved that, for every measurable set $A\subseteq \Omega$, the set $\bvphi(A)$ is also  measurable and  
\begin{equation}\label{MM-1}
|\bvphi(A)|\le C\,|A|^{1-\frac{3}{p}} \|\nabla\bvphi\|_{L^p(A)}^3.
\end{equation}
In particular, $|\bvphi(A)|=0$ for all $A\subset \bar\Omega$ with $|A|=0$
(that is, $\bvphi$ satisfies the so-called {\bf Lusin (N) property}). Furthermore, the following change of variable formula (or area formula) holds (see also \cite{Hajlasz}): if $g\ge 0$ is measurable, then
\begin{equation}\label{covf}
\int_{\bvphi(\Omega)} N(\bvphi,\by) g(\by)\,d\by=\int_\Omega g(\bvphi(\bx))  \det\nabla \bvphi(\bx)\,d\bx,
\end{equation}
 where \begin{equation}N(\bvphi,\by)=\mathcal H^0(\bvphi^{-1}(\by))=\#\{\bx\in\Omega\colon \bvphi(\bx)=\by\}. \label{N} \end{equation}

Consequently, the injectivity of  $\bvphi$ corresponds to the condition $N(\bvphi,\by)=1$,  for all $\by\in\bvphi(\Omega)$. In the context of nonlinear elasticity,  this condition has been addressed by Ball \cite{ball81},  for pure displacement one-to-one boundary conditions, and also by Ciarlet and Nec\v as \cite{ciarlet-necas1987},   in terms of  the inequality:
\begin{equation}\label{1-1}
\int_{\Omega}\det\nabla\bvphi(\bx)\,d\bx\leq |\bvphi(\Omega)|.
\end{equation}

The following result is useful for our purpose.

\begin{lemma}\label{lem-1} Let $\bvphi\in W^{1,p}(\Omega,\mathbb{R}^3)$ satisfy conditions (\ref{incompressible}) and (\ref{1-1}). Then,  $N(\bvphi,\by)=1$ for a.e.\,$\by\in\bvphi(\Omega);$ hence, for all measurable  $g\ge 0$, it follows that
\begin{equation}\label{covf-0}
\int_{\bvphi(\Omega)} g(\by)\,d\by=\int_\Omega g(\bvphi(\bx))  \det\nabla \bvphi(\bx)\,d\bx.
\end{equation}
 Moreover,  $|\bvphi^{-1}(B)|\le |B|/\delta_0$ for all measurable sets $B\subseteq \bvphi(\Omega)$.
In particular, $|\bvphi^{-1}(B)|=0$ for all $B\subset \bvphi(\Omega)$ with $|B|=0.$
\end{lemma}
\begin{proof} With $g(\by)\equiv 1$ in (\ref{covf}), by (\ref{1-1}), we obtain that
\[
|\bvphi(\Omega)|\le \int_{\bvphi(\Omega)} N(\bvphi,\by) \,d\by=\int_\Omega \det\nabla \bvphi(\bx)\,d\bx\leq |\bvphi(\Omega)|.
\]
Therefore, $N(\bvphi,\by)=1$ for a.e.\,$\by\in\bvphi(\Omega)$. Hence (\ref{covf-0}) follows from (\ref{covf}). In (\ref{covf-0}), let $g(\by)=\chi_B(\by)$, where $B\subseteq \bvphi(\Omega)$ is any measurable set, and 
$\chi_B$ denotes the characteristic function. So, 
we have
\[
|B|=\int_B \,d\by=\int_\Omega \chi_B(\bvphi(\bx))\det\nabla\bvphi\,d\bx =\int_{\bvphi^{-1}(B)} \det\nabla\bvphi\,d\bx \ge \delta_0 |\bvphi^{-1}(B)|,
\]
which completes the proof. \end{proof}

\begin{remark}
From the arguments of the  proof, we see  that  for maps $\bvphi\in W^{1,p}(\Omega,\mathbb{R}^3)$ with $\det \nabla\bvphi(\bx)>0$ a.e.\,$\Omega$,  if inequality  (\ref{1-1}) holds then it must be an equality. 
\end{remark}

We also need the following result on invertibility  of the deformation map.  For related results under weaker regularity conditions, we refer the reader to \cite{Fonseca-Gangbo}.
\begin{lemma}\label{lem-1-1} Let $\bvphi\in W^{1,p}(\Omega,\mathbb{R}^3)$ satisfy conditions (\ref{incompressible}) and (\ref{1-1}). Then, there exist sets $N\subset \Omega$ and $Z\subset \bvphi(\Omega)$ with $|N|=|Z|=0$ such that $\bvphi\colon\Omega\setminus N\to \bvphi(\Omega)\setminus Z$ is bijective and the  inverse map $\bvphi^{-1}=\bvpsi$ belongs to $ W^{1,p/2}(\bvphi(\Omega)\setminus Z)$, with
\begin{equation}\label{weak-grad}
\nabla_{\by} \bvpsi(\by)=((\nabla_{\bx}\bvphi)^{-1} \circ \bvpsi)(\by)=\frac{\adj ((\nabla\bvphi)\circ \bvpsi)(\by)}{\det((\nabla\bvphi)\circ \bvpsi)(\by)}\quad a.e.\,\,\by\in \bvphi(\Omega)\setminus Z. 
\end{equation}
\end{lemma}
\begin{proof} Let $N(\bvphi,\by)=1$ for $\by\in \bvphi(\Omega)\setminus Z_1$, where $|Z_1|=0.$  By \cite[Theorem 3.1]{Fonseca-Gangbo}, there exists  a set $N_1\subset\Omega$,  with $|N_1|=0$,  such that for each $\bx_0\in \Omega\setminus N_1$,  there is an open neighborhood $D_{\bx_0}\subset\subset \Omega$, for which there exist  $\delta>0 $ and a function $\bvpsi_0\in W^{1,p/2}(B_\delta(\by_0), D_{\bx_0})$,   with $\by_0=\bvphi(\bx_0)$, satisfying the  following properties:
\begin{eqnarray*}
&\bvpsi_0\circ \bvphi (\bx)=\bx \,\,\, a.e.\,\,\bx\in D_{\bx_0},\quad \bvphi\circ \bvpsi_0(\by)=\by \,\,\,  a.e.\,\by\in B_\delta(\by_0),&\\
&\nabla_{\by} \bvpsi_0 (\by)=((\nabla_{\bx}\bvphi)^{-1} \circ \bvpsi_0)(\by)\quad a.e.\,\,\by\in B_\delta(\by_0).&
\end{eqnarray*}
Let $N=N_1\cup \bvphi^{-1}(Z_1)\subset \Omega$ and $Z=Z_1\cup \bvphi(N_1).$ By Lemma \ref{lem-1},  $|Z|=|N|=0, $ and $\bvphi\colon\Omega\setminus N\to \bvphi(\Omega)\setminus Z$ is bijective. Let  $\bvpsi=\bvphi^{-1}\colon \bvphi(\Omega)\setminus Z\to \Omega\setminus N$ denote the inverse map. Let $\by_0\in \bvphi(\Omega)\setminus Z$ and $\bx_0=\bvpsi(\by_0)\in \Omega\setminus N; $ so $\bx_0\in \Omega\setminus N_1$ and $\by_0=\bvphi(\bx_0)$. Let $\bvpsi_0\in W^{1,p/2}(B_\delta(\by_0),D_{\bx_0})$ be the function determined above.   It follows that $\bvpsi=\bvpsi_0$ a.e.\,$B_\delta(\by_0).$ Hence $\bvpsi$ is weakly differentiable on $\bvphi(\Omega)\setminus Z$ and the weak gradient $\nabla\bvpsi(\by)$ is given by (\ref{weak-grad}). This also proves that $\nabla\bvpsi\in L^{p/2}(\bvphi(\Omega)\setminus Z)$, since $\adj\nabla\bvphi(\bx)\in L^{p/2}(\Omega)$ and $\det\nabla\bvphi(\bx))\ge\delta_0>0.$ Finally, since clearly $\bvpsi \in L^\infty(\bvphi(\Omega)\setminus Z),$  it follows that $\bvpsi\in W^{1,p/2}(\bvphi(\Omega)\setminus Z).$ This completes the proof. 
\end{proof}

In order to define the admissible set $\mathcal A$ of the variational problem, we first need to introduce  the functional $\mathcal S$. In particular, it will help us identify types of boundary conditions of the displacement field that are compatible with the injectivity conditions previously discussed.

Let $S\colon W^{1,p}(\Omega,\mathbb{R}^3)\to\mathbb{R}^+$ be a given functional such that $S(t\bvphi)=|t|S(\bvphi)$ for all $t\in\mathbb{R}$ and $\bvphi$. Assume that $S$ is continuous under the weak convergence of $W^{1,p}(\Omega,\mathbb{R}^3)$, and that if $\bvphi$ is constant and $S(\bvphi)=0$ then $\bvphi=0.$ Then, by the Sobolev-Rellich-Kondrachov compact embedding of $W^{1,p}(\Omega)\to L^p(\Omega)$, one easily has the following Poincar\'e-type inequality: there exists a constant $C$ such that
\begin{equation}\label{poincare-1}
\|\bvphi\|_{W^{1,p}(\Omega)}\le C (S(\bvphi) +\|\nabla\bvphi\|_{L^p(\Omega)})\quad\forall\,\bvphi\in W^{1,p}(\Omega,\mathbb{R}^3).
\end{equation}
For such a functional $S$ and given a constant $\beta>0$,  let
\begin{equation}
 \mathcal D_{S,\beta}(\Omega)=\{\bvphi\in W^{1,p}(\Omega,
\mathbb{R}^3): \, S(\bvphi)\le \beta\}. \label{D-Sbeta}
\end{equation}
In many applications, the functional $S$ can be chosen as one of the following:
\begin{eqnarray}
& S_1(\bvphi)=\|\bvphi|_{\partial\Omega}\|_{L^p(\partial\Omega)},  \mbox{ (Dirichlet boundary);}&\\
&S_2(\bvphi)=|\int_D \bvphi d\bx|, \,\, \mbox{where $D\subseteq \Omega$ with $|D|>0$,  \,\, (partial average);}&\\
&S_3(\bvphi)=|\bvphi(\bx_0)|, \,\, \mbox{where $\bx_0\in\bar\Omega$ is given,  \,\, (one-point),}&\label{s-3}
\end{eqnarray}
the last choice following from the compact embedding $W^{1,p}(\Omega)\to C(\bar\Omega)$ as $p>3.$

Assume $\mathcal B(\Omega)$ is any nonempty subset of $\mathcal D_{S,\beta}(\Omega)$ that is closed under the weak convergence of $W^{1,p}(\Omega,\mathbb{R}^3).$ We then introduce  an admissible class for energy $\mathcal E$ by
\begin{equation}
 \mathcal A =\{(\bvphi,Q): \mbox{$\bvphi\in  \mathcal B(\Omega)$ satisfies (\ref{incompressible}), (\ref{1-1})},\,
\, Q\in W^{1,r}(\bvphi(\Omega),\mathcal Q)\},  \label{Admissible-general}
\end{equation}
 where $\mathcal Q$ is defined by expression (\ref{set-Q}).

 \medskip

\noindent {\bf  Remarks. \,} 1. Examples of sets  $\mathcal B(\Omega)$ include the following ones, associated with standard types of  boundary conditions: 
\begin{eqnarray*}
&\mathcal B_1(\Omega)=\{\bvphi\in W^{1,p}(\Omega,
\mathbb{R}^3): \, \bvphi|_{\partial\Omega}=\bvphi_0\}\quad \mbox{($\bvphi_0$ a given trace-function),}&\\
&\mathcal B_2(\Omega)=\{\bvphi\in W^{1,p}(\Omega,
\mathbb{R}^3): \, \int_D \bvphi d\bx=0\} \quad  \mbox{($D\subseteq \Omega$ with $|D|>0$)},&\\
&\mathcal B_3(\Omega)=\{\bvphi\in W^{1,p}(\Omega,
\mathbb{R}^3): \, \bvphi|_A=\bvphi_0|_A\}\quad  \mbox{($A\subseteq \bar\Omega$ nonempty, $\bvphi_0$ bounded).}&
\end{eqnarray*}

2. The set $\mathcal B_3(\Omega)$   includes  $\mathcal B_1(\Omega)$ and  the case that   partial Dirichlet boundary conditions are prescribed.  $\mathcal B_3$ can be considered  a subset of $\mathcal D_{S,\beta}$ with $S=S_3$ defined by (\ref{s-3}).  This is mainly due to  the assumption $p>3$ and the compact embedding $W^{1,p}(\Omega)\to C(\bar\Omega)$.

3.  For $\mathcal B_1(\Omega)$ and  if $\bvphi_0|_{\partial\Omega}\in W^{1,p}(\partial\Omega)$  is injective, then $\bvphi\colon\Omega\to \bvphi_0(\Omega)$ is bijective \cite{ball81}. In this case, the current domain $\bvphi(\Omega)$ is fixed. The minimization problem considered below becomes less technical, with no need of  changing the Landau-de Gennes  energy integral  to the reference domain $\Omega.$

4. Deformation maps corresponding to $\mathcal B_2$ include anti-plane shear deformation. 
\medskip

Note that, for all $(\bvphi,Q)\in\mathcal A$, 
\[
-1/3<\lamin(Q(\by))\le {2}/{3}\quad a.e. \,\, \by\in \bvphi(\Omega).
\]
 Hence,
\begin{equation}\label{infinity-norm-Q}
\det (Q(\by)+\frac13 I) >0 \,\,\, a.e.\,\by\in \bvphi(\Omega),\quad \|Q\|_{L^\infty(\bvphi(\Omega))}\le \frac{2}{\sqrt3}.  
\end{equation}

We now  consider the {\it pull-back}  order,  step-length and effective deformation tensors:
\begin{equation}
\tilde Q(\bx)=Q(\bvphi(\bx)), \,\, \tilde L(\bx)= a_0(\tilde Q(\bx) +\frac13 I),\,\, G(\bx)=\tilde L(\bx)^{-\frac12} \nabla\bvphi(\bx). \label{Q-inv}
\end{equation}
We first explore the relationship between $Q(\by)$ and $\tilde Q(\bx)$. 

\begin{lemma}\label{lem-2}  Let $p>3$. Suppose that $r$ and $q$ are given by   (\ref{cond-r}) and  in (\ref{q}), respectively.  Then,   for all  $(\bvphi, Q)\in \mathcal A$,  $\tilde Q\in W^{1,q}(\Omega,\mathcal Q).$
Moreover,
\begin{equation}\label{weak-der}
\nabla_{\bx} \tilde Q(\bx)=\nabla_{\by} Q(\bvphi(\bx))\nabla\bvphi(\bx)\quad a.e.\,\,\bx\in\Omega.
\end{equation}
\end{lemma}
\begin{proof}
Clearly $\tilde Q\in L^\infty(\Omega).$ In general, the composition $\tilde Q=Q\circ \bvphi$ may not be weakly differentiable. However, we show that this is the case if $(\bvphi,Q)\in\mathcal A.$ By the approximation theorem \cite[Lemma 10]{Hajlasz}, there exist measurable sets  $Z_j\subseteq \bvphi(\Omega)$ and Lipschitz functions $Q_j$ in $\mathbb{R}^3$,  such that  $|Z_j|\to 0$ and $Q=Q_j$ in $\bvphi(\Omega)\setminus Z_j$. Let $\Omega_j=\bvphi^{-1}(\bvphi(\Omega)\setminus Z_j).$ Then $\tilde Q(\bx)=Q_j(\bvphi(\bx))$ for $\bx\in \Omega_j.$ Since $Q_j$ is Lipschitz, and  so $Q_j(\bvphi)$ is weakly differentiable on $\Omega$, it follows that $\tilde Q$  is weakly differentiable on $\Omega_j$, with weak gradient given by
\[
\nabla_{\bx} \tilde Q(\bx)=\nabla_{\by} Q_j(\bvphi(\bx))\nabla\bvphi(\bx)=\nabla_{\by} Q(\bvphi(\bx))\nabla\bvphi(\bx)\quad a.e.\,\,\bx\in\Omega_j.
\]
Note that $\Omega\setminus \Omega_j=\bvphi^{-1}(Z_j).$ 
Hence, $|\Omega\setminus\Omega_j|=|\bvphi^{-1}(Z_j)|\le \frac{1}{\delta_0}|Z_j|\to 0$. This  proves the weak differentiability of $\tilde Q$ on $\Omega$ and establishes the equation (\ref{weak-der}). Moreover, by relations (\ref{incompressible}) and (\ref{covf-0}),
\begin{equation}\label{coer-1}
\begin{split} \int_\Omega |\nabla_{\by}Q(\bvphi(\bx))|^rd\bx &\le \frac{1}{\delta_0} \int_\Omega |\nabla_{\by}Q(\bvphi(\bx))|^r\det\nabla\bvphi(\bx) d\bx\\
&=\frac{1}{\delta_0}\int_{\bvphi(\Omega)} |\nabla_{\by}Q(\by)|^r d\by.
\end{split}
\end{equation}
So, $\nabla_{\by}Q(\bvphi(\bx))\in L^r(\Omega).$ Since $\nabla\bvphi\in L^p(\Omega)$,  it follows  that 
\[
\mbox{$\nabla_{\bx}\tilde Q(\bx)=\nabla_{\by}Q(\bvphi(\bx))\nabla\bvphi(\bx)\in L^q(\Omega)$, where $\frac{1}{r}+\frac{1}{p}=\frac{1}{q}$ and $ q>1.$}
\]
Hence $\tilde Q\in W^{1,q}(\Omega,\mathcal Q)$. Furthermore,  it follows that 
\begin{equation}\label{coer-2}
\|\tilde Q\|_{W^{1,q}(\Omega)}\le C(\|\nabla_{\by}Q\|_{L^r(\bvphi(\Omega))} \|\nabla\bvphi\|_{L^p(\Omega)} +1),\end{equation}
where $C$ is a constant independent of $(\bvphi,Q)\in\mathcal A.$
\end{proof}
Let us now rewrite the total energy (\ref{total-energy-Lag_Eul}) for fields $(\bvphi, Q)$ in the admissible set. 

It follows from  (\ref{weak-der}) that,  for  $(\bvphi,Q)\in \mathcal A$, $\nabla_{\by} Q(\bvphi(\bx))=\nabla_{\bx} \tilde Q(\bx) \nabla\bvphi(\bx)^{-1}$. Therefore,  by  (\ref{covf-0}), we have  that, 
\begin{equation}
\mathcal E(\bvphi,Q)= \int_\Omega  \hat W(G)\,d\bx +  \int_{\Omega}\left (\mathcal L({\nabla}_{\bx} \tilde Q \nabla\bvphi^{-1}, \tilde Q) + f(\tilde Q)\right )\det\nabla \bvphi \,d\bx,  \label{total-energy-Lag_Eul-1}
\end{equation} for all $(\bvphi,Q)\in \mathcal A$.

Let us now discuss appropriate boundary conditions to impose on $Q$, either in the form of strong anchoring (Dirichelt) or by modifying the total energy by adding a surface energy penalty of the  Rapini-Papoular form \cite{Rapini-Papoular}.  For this, we first assume that $\bvphi$ satisfies Dirichlet boundary conditions on $\Gamma\subseteq\partial\Omega$ (Remarks 1 and 3, page 10),  and require one of the following  on $Q$:

\begin{enumerate}
\item 
Dirichlet boundary conditions:
\begin{eqnarray}
&&Q(\by)=Q_0(\bx), \quad \bx\in\Gamma\subseteq\partial\Omega \quad \textrm {with}, \label{Dirichlet-Q}\\
&& \by=\bvphi_0(\bx), \quad \bx\in\Gamma. \label{Dirichlet-phi}
\end{eqnarray}
\item  We include a surface energy contribution in the total energy. 
This term, of the Rapini-Papoular form is
\begin{equation}
\mathcal E_{\textrm{\tiny{S}}}(\bvphi, Q)=\int_{\Gamma}h(\bx, \tilde Q(\bx))\,dS, \label{surface-energy-model}
\end{equation}
where $\tilde Q|_{\Gamma}$ is the $L^q-$trace of $\tilde Q\in W^{1,q}(\Omega, \mathcal Q)$, $\alpha\geq 0$ and $h\geq 0$ denotes  a continuous function satisfying 
\begin{equation}
|h(\bx, A) -h(\bx, B)|\leq \alpha|A-B|^q,
\end{equation}
with $A, B\in{\mathbb M}^{3\times 3}$ and $\tr A=0=\tr B$.  This energy form includes  relevant expressions of the liquid crystal theory, such as
\begin{eqnarray}
&&h(\bx, Q)=  \tr(Q-Q_0)^2, \label{surface-energy}
\end{eqnarray}
with $Q_0\in\mathcal Q$ prescribed. 
\end{enumerate}
In the case that $\bvphi$ does not satisfy Dirichlet boundary conditions on any part of $\partial\Omega$,  we still allow for the modification of the energy as in  (\ref{surface-energy-model}), with $\Gamma\subseteq\partial\Omega$. We point out that the new energy integral is taken on the boundary of the reference domain and it involves the pull-back tensor $\tilde Q$. Additional regularity  is required to pose that energy integral in the current domain; this issue will not be addressed in the current work. 

The next two results will be employed in the proof of existence of energy minimizer. The first one establishes properties of weak limits of sequences of  effective deformation tensors and their relation with 
those of the corresponding deformation gradient tensors. The second lemma refers to the preservation of relation (\ref{1-1}) under weak convergence. 

\begin{lemma}\label{weak-product} Let $L_k\in L^\infty(\Omega,\mathbb{S}^3_+)$ with $\|L_k\|_\infty\le C.$ Let $G_k=L_k^{-1/2} F_k.$ Suppose $G_k\wconv \bar G$, $F_k\wconv \bar F$ in $L^p(\Omega,\mathbb{M}^3)$, where $p>3.$ Then, via a subsequence, $\adj G_k\wconv \bar H$, $\adj F_k\wconv \bar K$ in $L^{p/2}(\Omega,\mathbb{M}^3)$ and $\det G_k\wconv \bar g$, $\det F_k\wconv \bar f$ in $L^{p/3}(\Omega).$  Moreover, if $L_k(\bx)\to \bar L(\bx)$ for a.e.\,$\bx\in \Omega$, then 
\[
\bar L^{1/2}\bar G=\bar F,\quad \bar K= \bar H\adj(\bar L)^{1/2}, \quad (\det \bar L^{1/2}) \,\bar g=\bar f \quad a.e.\,\,\Omega.
\]
 \end{lemma}
 \begin{proof} Clearly, if $\{G_k\}$ is bounded in $L^p(\Omega)$, then $\{\det G_k\}$ and $\{\adj G_k\}$ are bounded in $L^{p/3}(\Omega)$ and $L^{p/2}(\Omega)$, respectively. Hence the weak convergence of a further subsequence of both sequences follows as $p>3.$ By the bounded convergence theorem, our assumption implies that  $g(L_k)\to g(\bar L)$ strongly in $L^q(\Omega)$, for all $q\ge 1$ and for all  continuous functions  $g$. Therefore, $F_k=L_k^{1/2} G_k\wconv \bar L^{1/2}\bar G$ in $L^1(\Omega), $ and  $\bar L^{1/2}\bar G=\bar F.$ The  two remaining statements follow from the elementary matrix identities: $\det (AB)=\det A \det B$ and $\adj(AB)=\adj(B)\adj (A).$
 \end{proof}
 
\begin{lemma}\label{lem-3} Let $\bvphi_k,\bvphi\colon\bar\Omega\to \mathbb{R}^3$ be continuous. Suppose $|\bvphi(\partial\Omega)|=0$  and $\bvphi_k\to \bvphi$ uniformly on $\Omega$ as $k\to\infty.$ Then
\begin{equation}\label{1-1-0}
\limsup_{k\to\infty} |\bvphi_k(\Omega)|\le |\bvphi(\Omega)|.
\end{equation}
Furthermore,  relation (\ref{1-1}) remains invariant under  weak convergence in $W^{1,p}(\Omega,\mathbb{R}^3)$,  for $p>3.$
\end{lemma}
\begin{proof} Without loss of generality, let us  assume that $\lim_{k\to\infty} |\bvphi_k(\Omega)|$ exists. So, for each $\epsilon>0$, there exists a positive integer $N$ such that
\[
|\bvphi_k(\bx)-\bvphi(\bx)|<\epsilon, \quad \forall\,k\ge N,\quad \forall\,\bx\in \bar\Omega.
\]
Let us denote $S=\bvphi(\Omega).$ The previous  inequality implies that $\bvphi_k(\Omega)\subseteq S_\epsilon$, for all $k\ge N,$ where $S_\epsilon=\{\by\in\mathbb{R}^3\colon \dist(\by, S)<\epsilon\}.$ Hence $|\bvphi_k(\Omega)|\le |S_\epsilon|$, for all $k\ge N.$ So
\begin{equation}\label{eq-11}
\lim_{k\to \infty} |\bvphi_k(\Omega)|\le |S_\epsilon|.
\end{equation}
Since $S_\epsilon\subseteq S_\delta$ for all $0<\epsilon<\delta$, and $\cap_{\epsilon>0}S_\epsilon=\bar S$, we have
\[
\lim_{\epsilon\to ^+} |S_\epsilon|=|\cap_{\epsilon>0} S_\epsilon|=|\bar S|=|S|+|\partial S\setminus S|.
\]
Since $S=\bvphi(\Omega)$, we easily verify that $\partial S\setminus S\subseteq \bvphi(\partial\Omega)$ and hence  $|\partial S\setminus S|=0.$ This proves that  $\lim_{\epsilon\to 0^+} |S_\epsilon|=|S|$, and thus (\ref{1-1-0}) follows  from (\ref{eq-11}). To show that condition (\ref{1-1}) holds under  weak convergence in $W^{1,p}(\Omega,\mathbb{R}^3)$, let us assume that $\bvphi_k\wconv  \bvphi\in W^{1,p}(\Omega,\mathbb{R}^3)$ and  also that
\[
\int_\Omega \det\nabla\bvphi_k\,d\bx \le |\bvphi_k(\Omega)|\quad\forall\, k=1,2,\cdots.
\]
A classical weak continuity theorem on determinants \cite{ball77} asserts that $\det\nabla\bvphi_k\wconv \det\nabla\bvphi$ in $L^{p/3}(\Omega)$ as $p>3.$ Also, by the compact embedding of $W^{1,p}(\Omega)\to C(\bar\Omega)$, it follows that $\bvphi_k\to \bvphi$ uniformly on $\Omega.$  Taking  limits as $k\to\infty$ in the above inequality and applying (\ref{1-1-0}),  we obtain  relation (\ref{1-1}) for $\bvphi.$ This concludes the proof of the  lemma.
\end{proof}

\section{Energy minimization}
The following theorem establishes the  existence of minimizer of the energy $\mathcal E$ in the  admissible class $\mathcal A$ defined above. 
\begin{theorem} \label{thm:dgl-incomp}
 Let the admissible set $\mathcal A$ be defined as in (\ref{Admissible-general}). Suppose there exists a pair $(\bvphi,Q)\in \mathcal A$ such that $\mathcal E(\bvphi,Q)<\infty.$
 Then, there exists at least one pair $({\bvphi}^*, Q^*)\in \mathcal A$  such that
\begin{equation}
\mathcal E(\bvphi^*, Q^*)=\inf_{(\bvphi, Q)\in \mathcal A}
\mathcal E(\bvphi, Q).
\end{equation}
\end{theorem}
\begin{proof}
By assumption,  there exists a constant
$K_1>0$ such that
\begin{equation}
0\le \inf_{(\bvphi, Q)\in \mathcal A} \mathcal{E}(\bvphi,Q)  < K_1.
\label{infupperbound}
\end{equation}
Let $(\bvphi_k,  Q_k)\in\mathcal A $  be a minimizing sequence for
$\mathcal E$, that is
\begin{equation}\lim_{k\to\infty} \mathcal E(\boldsymbol \varphi_k,  Q_{k})= \inf_{(\bvphi,  Q)\in\mathcal
A}\,\mathcal E(\bvphi,Q)<K_1.
\end{equation}
Denote $F_k, G_k, L_k, \tilde Q_k,\tilde L_k$ the corresponding quantities defined from $(\bvphi_k,Q_k)$ as above.

\medskip

\noindent {\bf Step 1: \,Coercivity.}
Given any $(\bvphi,Q)\in\mathcal A$, since $0<\lamax(\tilde L(\bx))\le a_0,$ we easily see that
\begin{equation}\label{lemma:G-adjG}
|\nabla\bvphi|\le \sqrt{a_0}\, |\tilde L^{-1/2}\nabla \bvphi|=\sqrt{a_0}\, |G|.
\end{equation}
 By the coercivity assumptions,  
\begin{equation}
\mathcal E(\bvphi, Q)
 \geq \alpha\int_{\Omega} |G|^p \,d\bx+ \int_{\bvphi(\Omega)} \big(f(Q)+\kappa |\nabla_{\by} Q|^{r}\big)\,d\by. \label{total-energy-coercivity}
\end{equation}
By (\ref{lemma:G-adjG}), we have 
\[
\int_{\bvphi(\Omega)} |\nabla_{\by}Q|^r\,d\by + \int_{\Omega} |\nabla\bvphi|^p \,d\bx\le C\, \mathcal E(\bvphi,Q).
\]
Since $\bvphi\in \mathcal D_{S,\beta} (\Omega)$ (as defined in (\ref{D-Sbeta})), applying (\ref{poincare-1}) and (\ref{coer-2}) yields
\begin{equation}\label{coer-3}
\|G\|_{L^p(\Omega)}+\|\bvphi\|_{W^{1,p}(\Omega)} +\|\tilde Q\|_{W^{1,q}(\Omega)}\le C(\mathcal E(\bvphi,Q)),
\end{equation}
where $C(M)$ is a constant depending on $M>0$ that remains bounded for $M$ in a bounded set.

\noindent{\bf Step 2: Convergence and compactness.} Let $(\bvphi_k,Q_k)\in\mathcal A$ be the minimizing sequence above. By (\ref{coer-3}),  we obtain that
\begin{equation}\label{estimates-all}
\{(G_k,  \bvphi_k, \tilde Q_k)\}\mbox{ is bounded in } \, L^p(\Omega)\times W^{1,p}(\Omega)\times W^{1,q}(\Omega).
\end{equation}
Hence,  there exists a  subsequence (labeled also by $k$) such that
\begin{eqnarray}
&G_k\wconv G^* \,\, \textrm {in}\,\, L^{p}(\Omega),\quad \bvphi_k\wconv \bvphi^* \,\, \textrm {in}\,\, W^{1,p}(\Omega),\quad  \tilde Q_k\wconv {\tilde Q}^* \,\, \textrm{in }\,\, W^{1,q}(\Omega),&\\
&g_k\equiv \det G_k \wconv g^* \,\,\mbox{ in $L^{p/3}(\Omega)$,}\quad H_k\equiv \adj(G_k)\wconv H^* \,\, \textrm {in}\,\, L^{{p}/{2}}(\Omega),& \label{weak1}\\
 &\det \nabla\bvphi_k \wconv \det\nabla\bvphi^*\,\, \mbox{ in $L^{p/3}(\Omega)$,} \quad \adj \nabla\bvphi_k\wconv \adj \nabla\bvphi^*\,\, \textrm {in}\,\, L^{{p}/{2}}(\Omega).& \label{weak2}
\end{eqnarray}
Since $p>3$, the convergence statements in (\ref{weak2}) follow from the classical weak continuity theorem of null-Lagrangians  \cite{ball77}.

We first prove  the following result on compensated compactness.
\begin{lemma}\label{cc}
Let $p>3$ and $q\in \mathbf R$ be as in (\ref{q}), and define  $m=\frac{pq}{p+2q}.$ Then $m>1$  and
\begin{equation}\label{weak-limit-equation-0}
\nabla_{\bx}\tilde Q_k \adj\nabla\bvphi_k \wconv \nabla_{\bx}\tilde Q^* \adj\nabla\bvphi^*\,\,\mbox{ in $L^m(\Omega)$.}
\end{equation}
\end{lemma}
\begin{proof} Since $\{\nabla_{\bx}\tilde Q_k\}$ is bounded in $L^q(\Omega)$ and $\{\adj \nabla\bvphi_k \}$ is bounded in $L^{p/2}(\Omega)$, we have that $\{\nabla_{\bx}\tilde Q_k \,\adj \nabla\bvphi_k\}$ is bounded in $L^{m}(\Omega)$, with  $m$  as above. 
 It follows from (]ref{q})  that $m>1.$  Therefore, we only need to show the convergence (\ref{weak-limit-equation-0}) in the distribution sense. Component-wise, such a convergence follows from the classical div-curl lemma, or from the representation  
\begin{equation}
\big((\nabla_{\bx}\tilde Q)(\adj \nabla\bvphi)\big)_{isj}= \sum_{m=1}^3\frac{\partial \tilde Q_{is}}{\partial x_m}(\adj \nabla\bvphi)_{mj}=\det(\nabla \mathbf u), \label{det-grad-adj}
\end{equation}
where, for fixed  $i,j,s=1, 2, 3$,  the vector-field  $\mathbf u:\Omega\to\mathbb{R}^3$  is defined with components $u_1,u_2,u_3$ given by $u_j= \tilde Q_{is}$ and $ u_k=\bvphi\cdot\mathbf e_k, \, \forall \,k\neq j.$
\end{proof}

Since $\|\bvphi_k\|_{W^{1,p}(\Omega)}$ and $\|Q_k\|_{W^{1,r}(\bvphi_k(\Omega))}$ are bounded and $p$, $r$ are both greater than 3, we have by Morrey's inequality,
\[
|\bvphi_k(\bx)-\bvphi_k(\bx')|\le C_1|\bx-\bx'|^{1-\frac{3}{p}},\quad \forall\,\bx, \,\bx'\in \bar\Omega,
\]
\[
|Q_k(\by)-Q_k(\by')|\le C_2|\by-\by'|^{1-\frac{3}{r}},\quad \forall\,\by, \,\by'\in \overline{\bvphi_k(\Omega)}.
\]
Hence
\[
|\tilde Q_k(\bx)-\tilde Q_k(\bx')|\le C_3 |\bx-\bx'|^{\gamma},\quad \forall\,\bx,\, \bx'\in\bar\Omega,
\]
where $\gamma=(1-\frac{3}{p})(1-\frac{3}{r})\in (0,1).$ This shows that both $(\bvphi_k)$ and $(\tilde Q_k)$ are uniformly bounded and 
equi-continuous on $\bar\Omega$; hence, by Arzela-Ascoli theorem,  it follows that, via another subsequence, $$\bvphi_k\to \bvphi_* \quad \textrm { and} \quad \tilde Q_k(\bx)\to \tilde Q_*(\bx),$$ uniformly on $\bar\Omega.$  Moreover, by   uniqueness of weak limit,  $\bvphi_*=\bvphi^*$ and  $\tilde Q_*=\tilde Q^*;$ that is, $\bvphi_k\to\bvphi^*$ and $\tilde Q_k=Q_k(\bvphi_k)\to  \tilde Q^*$ uniformly on $\bar\Omega.$ This imply that $\tilde L_k(\bx)\to a_0(\tilde Q^*(\bx)+\frac 13 I)\equiv \tilde L^*(\bx)$ uniformly on $\Omega.$ Hence,  by Lemma \ref{weak-product}, for a.e.\,\,$\bx\in \Omega$,
\begin{eqnarray}
    &\tilde L^*(\bx)^{1/2} G^*(\bx)=\nabla\bvphi^*(\bx),&\label{weak-limit-equation-1}\\
    & \adj \nabla \bvphi^*(\bx)=H^*(\bx) \adj(\tilde L^*(\bx))^{1/2}, &\label{weak-limit-equation-2}\\
& \det \nabla \bvphi^*(\bx)=g^*(\bx) \det(\tilde L^*(\bx))^{1/2}. &\label{weak-limit-equation-3}
\end{eqnarray}

 \noindent {\bf Step 3: Definition and properties of  $Q^*$.} 
By the weak convergence property $\det\nabla \bvphi_k\wconv \det \nabla\bvphi^*$, Lemma \ref{lem-3}  and the weak closedness of $\mathcal B(\Omega)$, we have that $\bvphi^*\in \mathcal B(\Omega)$ and it satisfies (\ref{incompressible}) and (\ref{1-1}). Hence, by Lemma \ref{lem-1-1}, there exists an inverse map \[
\bvpsi^*=(\bvphi^*)^{-1}\in W^{1,p/2}(\bvphi^*(\Omega)\setminus Z, \Omega\setminus N).
\]
This allows us to define the tensor map
\begin{equation}
Q^*(\by)=\tilde Q^*(\bvpsi^*(\by)), \quad  \by\in\bvphi^*(\Omega)\setminus Z. \label{define-Q*} 
\end{equation}
From  (\ref{MM-1}), it follows that  $|(\bvpsi^*)^{-1}(N_j)|=|\bvphi^*(N_j)|\to 0$ if $N_j\subset \bvphi^*(\Omega)\setminus Z$ and $|N_j|\to 0$. Therefore, we  follow the same  lines of proof as in Lemma \ref{lem-2} to conclude that $Q^*$ is weakly differentiable and 
\begin{equation}\label{chain-rule}
 \nabla_{\by}Q^*(\by)=\nabla_{\bx}\tilde Q^*(\bvpsi^*(\by))(\nabla_{\bx}\bvphi^*(\bvpsi^*(\by)))^{-1} \quad  a.e.\,\,\by\in\bvphi^*(\Omega)\setminus Z.
\end{equation}
This, together with the fact that  $\nabla\tilde Q^*\in L^q$ and $(\nabla\bvphi^*)^{-1}=\frac{\adj\nabla\bvphi^*}{\det\nabla\bvphi^*}\in L^{p/2}$ give
\begin{equation}\label{chain-rule-2}
Q^*\in W^{1,m}(\bvphi^*(\Omega)\setminus Z),\quad 
\end{equation}
with $m$ as in Lemma 4.2.

 We next show that ${\tilde Q}^*(\bx)\in\mathcal Q$ and thus $Q^*(\by)\in\mathcal Q$,  by studying  the properties of the minimizing sequence $\{\tilde Q_k\}$. 
 We may  assume $\tilde Q_k(\bx)\to \tilde Q^*(\bx)$ for a.e.\,$\bx\in\Omega.$ Since $\tilde Q_k(\bx)\in \mathcal Q$, it follows that 
\begin{equation}\label{Q-in-Q}
\tilde Q^*(\bx)\in {\mathcal Q}\quad \forall\,\bx\in\Omega.
\end{equation}
It remains to show  that  $\det (\tilde Q^*(\bx)+\frac13)>0$ a.e.\,on $\Omega$. For this, 
 let  us denote $q_k:=\det (\tilde Q_k+ \frac{1}{3}I).$  Hence,  the sequence  $\{q_k\}$ is bounded and $$q_k(\bx)\to q^*(\bx)=\det(\tilde Q^*(\bx)+\frac13), $$ a.e. \,$\bx\in\Omega.$ 
Now,  we want to prove that $q^*>0$ a.e. $\Omega$. For this, we argue by contradiction and  supposes that $q^*=0$ on a set $A\subset\Omega$, with $|A|>0.$  Note that
 $$ \int_A\det(\tilde Q_k+\frac{1}{3}I)\,d\bx\to  \int_A\det(\tilde Q^*+\frac{1}{3}I)\,d\bx=0.$$
Consider now the sequence $f_k(\bx):=f(\tilde Q_k(\bx))$.  Since $f_k\geq 0$, it follows from Fatou's theorem that, 
 $$\int_A \liminf_{k\to\infty} f_k(\bx)\,d\bx\leq  \liminf_{k\to\infty}\int_A f_k(\bx)\,d\bx.$$
 By the blow-up assumption (\ref{Phi3}) on $f$, for $\bx\in A$, 
  $$ \liminf_{k\to\infty} f_k(\bx)=  \lim_{\det(Q+\frac{1}{3}I)\to 0} f(Q)=+\infty,$$
 and consequently $ \lim_{k\to\infty}\int_A f(\tilde Q_k(\bx))\,d\bx= +\infty. $ This  contradicts the fact that $$\int_{\Omega} f(\tilde Q_k(\bx))d\bx \le C\mathcal E(\bvphi_k,Q_k)< CK_1.$$
  Hence, $\tilde Q^*(\bx)\in\mathcal Q $ a.e.\,$\bx\in \Omega$.  So $Q^*(\by)\in \mathcal Q$ for a.e.\,$\by\in\bvphi^*(\Omega).$ 

Finally, we define 
  \begin{equation}
  L^*(\by)=a_0 (Q^*(\by)+\frac{1}{3}I). \label{L*}
  \end{equation}
Then $\det L^*(\by)>0$; hence  $L^*(\by)=a_0 (Q^*(\by)+\frac{1}{3}I)$ is invertible a.e.\,$\by\in \bvphi^*(\Omega).$ 

 By (\ref{weak-limit-equation-1}) and (\ref{weak-limit-equation-2}), we have
 \[
    G^*=(\tilde L^*)^{-1/2} \nabla\bvphi^*,\,\,  H^* =\adj G^*,\,\, g^*=\det G^* \quad a.e.\,\,\Omega.
\]
Therefore, $\mathcal E(\bvphi^*,Q^*)=E_1^* +E_2^* +E_3^*,$
where
\begin{equation}
\label{split-energy}
\begin{split} 
E_1^* =\int_\Omega& \Psi(G^*, H^*, g^*)\,d\bx,\quad E_2^*=\int_\Omega f(\tilde Q^*)  \det\nabla\bvphi^* d\bx,\\
&E_3^* = \int_\Omega \mathcal L(\nabla \tilde Q^*(\nabla\bvphi^*)^{-1}, \tilde Q^*)\,\det\nabla\bvphi^* d\bx.
\end{split}
\end{equation}

  \medskip

  \noindent{\bf Step 4:  Weak lower semicontinuity of the energy.}  As in (\ref{split-energy}), we write $$\mathcal E(\bvphi_k,Q_k)=E_1^k +E_2^k +E_3^k,$$
where
\begin{equation}\label{weak-limit-equation-5}
\begin{split} 
E_1^k =\int_\Omega &\Psi(G_k, H_k, g_k)\,d\bx,\quad E_2^k=\int_\Omega f(\tilde Q_k)  \det\nabla\bvphi_k d\bx,\\
&E_3^k = \int_\Omega  \mathcal L(\nabla \tilde Q_k (\nabla\bvphi_k)^{-1}, \tilde Q_k)\,\det\nabla\bvphi_k d\bx.
\end{split}
\end{equation}
The weak lower semicontinuity of  $\mathcal E$ along $(\bvphi_k,Q_k)$ follows from that of each integral  functional of (\ref{weak-limit-equation-5}). Specifically,
\begin{enumerate}
\item The weak lower semicontinuity of  $E_1^k$
follows from the convexity of $\Psi(G,H,g)$ on $(G,H,g)$ and the weak convergence $G_k\wconv G^*$, $H_k\wconv H^*$ and $g_k\wconv g^*$ established above.
\item  The weak lower semicontinuity  of  $E_2^k$ 
 follows from the weak convergence of $\det\nabla\bvphi_k\wconv \det\nabla\bvphi^*$ and the pointwise convergence of $f(\tilde Q_k(\bx))\to f(\tilde Q^*(\bx))$, which is uniform on sets where $\lamin(\tilde Q^*(\bx))\ge -\frac13 +\tau$ for any given $\tau>0.$
\item Let us now prove the  weak lower semicontinuity of $E_3^k$.
 First, we  note that 
\[
\mathcal L(AB^{-1},Q)\det B =\mathcal F(A\adj B, \det B, Q),
\]
where $\mathcal F(A,t,Q)=t\mathcal L(A/t, Q)$. Elementary calculations show that $\mathcal F(A,t,Q)$ is convex in $(A,t)$, for $t>0$,  provided $\mathcal L(A,Q)$ is convex in $A$.  Let us  write
\begin{equation}\label{eq-final}
E_3^k=\int_\Omega \mathcal F(\nabla \tilde Q_k \adj \nabla\bvphi_k,\det\nabla\bvphi_k, \tilde Q_k)\, d\bx.
\end{equation}
The weak lower semicontinuity of $E_3^k$  follows from  the convexity of $\mathcal F(A,t,Q)$ with respect to $(A,t)$, the weak convergences of  $\nabla \tilde Q_k \adj \nabla\bvphi_k\wconv \nabla \tilde Q^* \adj \nabla\bvphi^*$ and $\det\nabla\bvphi_k\wconv \det\nabla\bvphi^*$,  and the uniform convergence of $\tilde Q_k\to\tilde Q^*.$
\end{enumerate}
So, the combination of these three steps establishes the weak lower semicontinuity of $\mathcal E$ on $\mathcal A$.
Hence, 
\[
\mathcal E(\bvphi^*,Q^*)\le \lim_{k\to \infty}\mathcal E(\bvphi_k,Q_k)<K_1<\infty.
\]
This, in turn, shows that
\[
\int_{\bvphi^*(\Omega)} |\nabla_{\by}Q^*(\by)|^r\,d\by <\infty.
\]
Consequently,  $Q^*\in W^{1,r}(\bvphi^*(\Omega),\mathcal Q)$, and  hence $(\bvphi^*,Q^*)\in\mathcal A$ is a minimizer of $\mathcal E$ on $\mathcal A.$ This completes the proof of Theorem \ref{thm:dgl-incomp}.
\end{proof}
In the last part of this section, we will address boundary conditions on $Q$, in cases compatible with the type of regularity assumed so far. 
Straightforward modifications of the above proof, allow us to establish the following results.

\begin{corollary}
 Let $p>3$, and $\sigma\geq 0$.  Suppose that  $\bvphi_0\in W^{1,p}(\partial\Omega, {\mathbb R}^3)$,  injective, and  $Q^0\in W^{1,2}(\partial\Omega, \mathcal Q)$ are prescribed. 
Let  $\mathcal E$ and $\mathcal E_{\textrm{\tiny{S}}}$ be as in (\ref{total-energy-Lag_Eul}) and (\ref{surface-energy}), respectively, and define
\begin{equation}
\mathcal E_{\textrm{\tiny{total}}}= \mathcal E+ \sigma\mathcal E_{\textrm{\tiny{S}}}. \label{E_bulk+E_surface}
\end{equation}
Suppose that $\mathcal A$ is as in (\ref{Admissible-general}) with $r\geq 2$ and  $\mathcal B=\mathcal B_1$. 
Then,  there exists a minimizer of the energy  (\ref{E_bulk+E_surface}), in each of the following cases:
\begin{enumerate}
\item $\sigma=0$ and $Q$ satisfies the boundary conditions (\ref{Dirichlet-Q}). 
\item $\sigma\neq 0$ and $\Gamma\subseteq\partial\Omega$.
\end{enumerate}
\end{corollary}
This corollary establishes existence of minimizer, even in the case that  the Landau-de Gennes energy is quadratic on gradient of $Q$, as for standard nematic liquid crystals.   
\begin{corollary} Suppose that $\Gamma\subset\partial\Omega$, with $\Gamma\neq\partial\Omega$.  Let $p$, $\sigma$,  $Q_0$, $\mathcal E$ and  $\mathcal E_{\textrm{\tiny{S}}}$ be as in Corollary 1. 
 Suppose that  $\bvphi_0\in W^{1,p}(\Gamma, {\mathbb R}^3)$ is  injective and (\ref{Dirichlet-phi}) holds.  Suppose that $\mathcal A$ is as in (\ref{Admissible-general}) with $\mathcal B=\mathcal B_3$. 
 Then,  there exists a minimizer of the energy  (\ref{E_bulk+E_surface}), in each of the following cases:
\begin{enumerate}
\item $\sigma=0$ and $Q$ satisfies the boundary conditions (\ref{Dirichlet-Q}). 
\item $\sigma\neq 0$. 
\end{enumerate}
\end{corollary}


Next, we give an example of a family of plane deformations in the framework of Corollary 2 and that  are relevant to   experimental settings of liquid crystal elastomers \cite{KundlerFinkelmann1995},  \cite{zubarev1999monodomain}, \cite{Finkelmann-Sanchez2008} and  \cite{Finkelmann-Sanchez2009}.  For prescribed constants $a>0, b>0, c>0$, and denoting 
$\mathcal D=\{(x_1, x_2):   -a<x_1<a, \,\, -b<x_2<b \},  $
 we define  the  domain and boundary, respectively,
\begin{eqnarray*}\Omega=\{\bx: (x_1, x_2)\in\mathcal D, \,\, -c< x_3<c\},\\ 
\Gamma_{\pm}=\{\bx: -a<x_1<a, \,  x_2=\pm b,\,  -c<x_3<c \}.
\end{eqnarray*}
Suppose that  $Q_0\in W^{1,2}(\Gamma, \mathcal Q)$,  $\lambda\geq 1$  and $\delta>0$ are also prescribed. Letting $\bx=(x_1, x_2, x_3)\in \Omega$, we consider the family of plane  deformations, 
\begin{eqnarray}
&& y_1= \varphi_1(x_1, x_2), \nonumber\\
&& y_2= \varphi_2(x_1, x_2), \nonumber\\
&& y_3= x_3, \label{plane-deformations}
\end{eqnarray}
subject to the constraint \,
$\det[\varphi_1, \varphi_2]\geq \delta, $
and boundary conditions,
\begin{eqnarray*}
&&\varphi_2(x_1, \pm b)= \pm \lambda b, \\
&&Q(x_1, \pm b)=Q_0.
\end{eqnarray*}
The following observations apply:
\begin{enumerate}
\item The deformations in (\ref{plane-deformations}) contain extensions, uniaxial and biaxial,  and shear. A modification of the third equation to the form $y_3=\alpha(x_1, x_2)+ \beta(x_1, x_2)x_3$, with appropriate choices of the continuous coefficients $\alpha$ and $\beta$, also allows to include extension or compression along $x_3$. 
\item The two-dimensional deformation map $(\varphi_1, \varphi_2)$ belongs to the family of Lipschitz continuous functions in $\mathcal D$. This implies the absence of cavities in the deformed configuration, and it also guarantees injectivity of the map.  
\item  A class of critical points of the energy of the form (\ref{plane-deformations}) are the stripped domains occurring as $\lambda>1$ reaches a critical value. These are domains parallel to the $x_1$-direction, with alternating positive and negative values of the shear rate.  An analysis of the length scales of these patterns is carried out in \cite{Calderer-domain-patterns2013}.

\end{enumerate}

\section{Conclusion}{\footnote{At the time of completion of this article, the authors became aware of a recent preprint studying a model of liquid crystal elastomer with director field gradients \cite{BarchesiDeSimone2013}. The model studied  here and the methods of proof show  fundamental differences with the aforementioned work, resulting in two independent articles, with different scope and research points of view.}}

In this article, we study existence of minimizers of a Landau-de Gennes liquid crystal elastomer. Special features of the work include the modeling of the liquid crystal behavior with the nematic order tensor $Q$, defined in the current configuration, rather than the customary uniaxial director field $\bn$. Also, the Landau-de Gennes energy is taken in the Eulerian frame of the liquid crystal. The passing between the Lagrangian frame of the elastic energy and the Eulerian one of the liquid crystal carries analytical difficulties due to the need to ensure invertibility of the deformation map. This is achieved for an appropriate  class of class of Dirichlet boundary conditions.  However, the analysis does not cover the case that an anchoring energy is prescribed on the deformed boundary of the domain.  This requires a higher regularity of the order tensor to ensure that boundary integrals of  the trace of  $Q^2$ are well defined. This is the subject of current analysis. 

\section*{Acknowledgments}  In this research, Calderer and Garavito  were partially supported by  grants from the National Science Foundation  NSF-DMS 1211896. Calderer also expresses her gratitude to the Newton Institute for Mathematical Sciences, University of Cambridge, UK,  where part of this research was carried out. Yan is also grateful for the support of Tang Ao-Qing visiting professorship from the College of Mathematics, Jilin University, China, where part of this research was carried out.



\medskip
\medskip

\end{document}